\documentclass[reqno]{amsart}

\usepackage{graphicx}  
\usepackage{tikz-cd}
\usepackage{bbm}

\usepackage{amsmath,amssymb,latexsym, amsfonts, amscd, amsthm} 

\usepackage{hyperref}
\hypersetup{colorlinks=false}
\usepackage{verbatim} 
\usepackage{xcolor}
\hypersetup{colorlinks=false,linkbordercolor=red,linkcolor=green,pdfborderstyle={/S/U/W 1}}

\usepackage[normalem]{ulem}
\usepackage{lipsum} 
\usepackage{fancyhdr}
\usepackage{lastpage}

\usepackage[english]{babel}

\usepackage{bbm}
\usepackage{enumerate}
\usepackage[all]{xy}

\usepackage{hyperref}
\hypersetup{colorlinks=false}

\allowdisplaybreaks
\newtheorem{thm}{Theorem}[section]
\newtheorem{pro}[thm]{Proposition}
\newtheorem{lm}[thm]{Lemma}

\newtheorem{que}{Question}

\numberwithin{equation}{section}
\newtheorem{defn}[thm]{Definition}

\theoremstyle{remark}

\theoremstyle{remark}
\newtheorem{rem}[thm]{Remark}

\DeclareMathOperator*{\Ext}{Ext}

\DeclareMathOperator*{\Irr}{Irr}

\DeclareMathOperator*{\Gal}{Gal}
\DeclareMathOperator*{\trace}{Tr}

\DeclareMathOperator*{\Ind}{\textsf{Ind}}

\DeclareMathOperator*{\Res}{Res}

\DeclareMathOperator*{\scn}{sc}

\DeclareMathOperator*{\Hom}{Hom}

\newcommand{\vp}{\varphi}

\newcommand{\s}{\simeq}

\def\cS{\mathcal S}

\newcommand{\si}{\sigma}

\newcommand{\CC}{\mathbb{C}}
\newcommand{\NN}{\mathbb{N}}

\newcommand{\QQ}{\mathbb{Q}}

\newcommand{\bH}{\mathbf H}

\def\bG{\bold G}

\DeclareMathOperator*{\SL}{SL}
\DeclareMathOperator*{\GL}{GL}

\DeclareMathOperator*{\PGL}{PGL}

\title[Distinguished Representations for $\SL_n(D)$]{Distinguished Representations for $\SL_n(D)$ where $D$ is a quaternion division algebra over a $p$-adic field}

\author{Kwangho Choiy}
\address{School of Mathematical and Statistical Sciences, Southern Illinois University, Carbondale, IL 62901-4408} 
\email{kchoiy@siu.edu}

\author{Shiv Prakash Patel}
\address{Department of Mathematics, IIT Delhi, Hauz Khas, New Delhi - 110016} 
\address{Department of Mathematics, IIT Dharwad, Chikkamalligawad Village, Dharwad - 580011}
\email{shivprakashpatel@gmail.com}

\subjclass[2010]{Primary: 22E50; Secondary: 11S37, 20G25, 22E35}
\keywords{distinguished representation, multiplicity formula, non-quasi-split inner form, local Langlands correspondence}

\date{\today}  

\begin{document}

\begin{abstract}
Let $D$ be a quaternion division algebra over a non-archimedean local field $F$ of characteristic zero. 
Let $E/F$ be a quadratic extension and $\SL_{n}^{*}(E) = {\GL}_{n}(E) \cap \SL_{n}(D)$. 
We study distinguished representations of $\SL_{n}(D)$ by the subgroup $\SL_{n}^{*}(E)$.
Let $\pi$ be an irreducible admissible representation of $\SL_{n}(D)$ which is distinguished by $\SL_{n}^{*}(E)$. We give a  multiplicity formula, i.e. a formula  for the dimension of the $\CC$-vector space ${\Hom}_{\SL_{n}^{*}(E)} (\pi, \mathbbm{1})$, where $\mathbbm{1}$ denotes the trivial representation of $\SL_{n}^{*}(E)$. 
This work is a non-split inner form analog of a work by Anandavardhanan-Prasad which gives a  multiplicity formula for $\SL_{n}(F)$-distinguished irreducible admissible representation of $\SL_{n}(E)$.
\end{abstract}

\maketitle

\section{Introduction} \label{intro}
Let $H$ be a subgroup of a group $G$. An irreducible  representation $\pi$ of $G$ is said to be $(H, \chi)$-distinguished for a character $\chi$ of $H$ if $\Hom_{H}(\pi, \chi) \neq 0$. In other words, there exists $\ell : \pi \rightarrow \CC$ such that 
\[
\ell(\pi(h)v) = \chi(h) \ell(v)
\]
for all $h \in H$ and $v \in \pi$. 
If $\chi = \mathbbm{1}$ is the trivial character of $H$ and $\pi$ is $(H, \mathbbm{1})$-distinguished then we simply say  $\pi$ is $H$-distinguished or $\pi$ is distinguished by $H$. 
There are two main question in this direction. First, identifying the irreducible representations $\pi$ of $G$ which are $(H, \chi)$-distinguished. Second, identify the spcae $\Hom_{H}( \pi, \chi)$ if it is non-zero, in particular its dimension. Our discussion will focus on the pairs $(G,H)$ for which the space $\Hom_{H}( \pi, \chi)$ is finite dimensional for all irreducible representations $\pi$ and characters $\chi$ of $H$.

Motivated by Anandavardhana-Prasad's works \cite{ap03, ap18}, we consider a generalized setting. 
Given a $p$-adic group $G,$ we let $G_\flat$ be a subgroup of $G$ such that
\begin{equation}  \label{cond on G}
[G_\flat,  G_\flat]=[G,  G] \subseteq G_\flat \subseteq G,
\end{equation}  
where $[A, A]$ stands for the derived group of $A.$ Let $H$ be a subgroup of $G$ and set 
\[
H_{\flat} = H \cap G_{\flat}.
\]

Our study in general is to investigate the dimension of the $\CC$-vector space
\[
{\Hom}_{H_\flat} (\pi_{\flat}, \mathbbm{1})
\]
for any irreducible admissible representation $\pi_{\flat}$ of $G_\flat.$ 
The further detail for this set up can be found in Section \ref{section set-up}.

Let $E$ be a quadratic extension of $F$ and let, $D$ be a quaternion division algebra over $F$.
Then, $H=\GL_n(E) \hookrightarrow G=\GL_n(D)$ and this embedding is unique up to a conjugation by elements in $\GL_n(D)$, by Skolem-Noether theorem. We fix an embedding of $H$ in $G$ and write $H \subset G$.
For any character $\chi : H \rightarrow \CC^{\times}$, it is known 
that the dimension of the $\CC$-vector space $ {\Hom}_{H} (\pi, \chi)$ is less than or equal to 1
for any irreducible admissible representation $\pi$ of $G$ \cite{lu23}. 
A conjecture of Prasad-Takloo-Bighash \cite[Conjecture~1]{ptb11}, that is now a theorem \cite[Theorem~1.2]{sx24}, 
gives a criterion for the representations $\pi$ of $G$ and characters $\chi$ of $H$ for which the dimension of the space $ {\Hom}_{{\GL}_{n}(E)} (\pi, \chi)$ equals 1.

Let $\mathcal{N} : M_{n}(D) \rightarrow F$ be the reduced norm map. 
Let $\SL_{n}(D) = \{ g \in {\GL}_{n}(D) : \mathcal{N}(g)=1 \}$ and $\SL_{n}^{*}(E) = {\GL}_{n}(E) \cap \SL_{n}(D)$. 
Note that $\SL_{n}(E) \subset \SL_{n}^{*}(E)$. 
In the paper, we study the distinction of an irreducible representation of $\SL_{n}(D)$ by the subgroup $\SL_{n}^{*}(E)$. 
More precisely, the main question is the following.
\begin{que}
Let $\pi_{\flat}$ is an irreducible admissible representation of $\SL_{n}(D)$. What is the dimension of the $\CC$-vector space
\[
{\Hom}_{\SL_{n}^{*}(E)} (\pi_{\flat}, \mathbbm{1}) ~~?
\]
\end{que}

The above question is a non-split analog to the question considered by Anadavardhanan-Prasad, \cite{ap03, ap18}, where they compute the dimension of the $\CC$-vector space ${\Hom}_{\SL_{n}(F)} (\pi_{\flat}, \mathbbm{1})$, where $\pi_{\flat}$ is an irreducible admissible representation of $\SL_{n}(E)$.
Moreover, Anandavardhana-Prasad also classify the representations of $\SL_{n}(E)$ that are distinguished by $\SL_{n}(F)$.

Now we let $\pi_\flat$ be an irreducible admissible representation of $\SL_n(D)$ which is $\SL_n^*(E)$-distinguished. 
In Proposition \ref{key pro}, we prove that there exists an irreducible admissible representation $\pi$ of ${\GL}_n(D)$ which is ${\GL}_n(E)$-distinguished such that 
\[
\pi_\flat \subset {\Res}^{{\GL}_n(D)}_{\SL_n(D)} (\pi).
\]
To materialize the non-trivial multiplicity occurring in this restriction unlike the split case, we bring in the dimension of the irreducible representation $\rho_{\pi_\flat}$ of the finite group $\cS_{\vp, \scn}$ living in $L$-groups corresponding $\pi_\flat$ via the local Langlands correspondence for $\SL_n(D)$ \cite{gk82, tad92, hs11}. 
The relevant details will be reviewed in Section \ref{restrictionsection}.

Let $\Irr(G)$ denote the set of the isomorphism classes of irreducible admissible complex representations of $G$.
The main result in the paper (Theorem \ref{mainthem}) is the following.
\begin{thm} \label{mainthm-intro}
Let $\pi_{\flat}$ be an irreducible admissible representation of $\SL_{n}(D)$ which is $\SL_{n}^{*}(E)$-distinguished. 
Let $\pi$ be an irreducible admissible representation of ${\GL}_{n}(D)$ which is ${\GL}_{n}(E)$-distinguished such that $\pi_{\flat} \subset {\Res}^{{\GL}_n(D)}_{\SL_n(D)} (\pi)$.
Denote by $\rho_{\pi_\flat} \in \Irr\big(\cS_{\vp, \scn}({^L}{\SL}_n(D)^\circ), ~ \zeta_{{\SL}_n(D)} \big)$ via \eqref{bij sl} the internal structure of $L$-packet of $\SL_n(D)$, (refer to Section \ref{restrictionsection}).
Let  $\mathcal{C}_{\pi}$ be the number of irreducible representations of ${\GL}_n(D)^+ := {\SL}_n(D) \cdot {\GL}_n(E)$ containing $\pi_\flat$ in ${\Res}^{{\GL}_n(D)}_{{\GL}_n(D)^+} \pi.$
Then, we have
\begin{equation} \label{mainequation-intro}
{\dim}_{\CC} {\Hom}_{\SL_n^*(E)} (\pi_\flat,\mathbbm{1}) = \frac{|X_\pi|}{|Z_\pi/Y_\pi|} \cdot \frac{{\dim} \rho_{\pi_\flat}}{\mathcal{C}_{\pi}^2},
\end{equation} 
where $X_{\pi}, Y_{\pi}, Z_{\pi}$ are given in Definition \ref{def XYZ}.
\end{thm}
Section \ref{mainresult} is dedicated to generalizing the machinery for the split case in \cite{ap18} into our non-split inner from setting. 
The proof can be summarized as follows. 
We start with the definition of  $X_\pi$ whose cardinality turns out to equal to the dimension of the $\CC$-vector space ${\Hom}_{\SL_n^*(E)}(\pi,\mathbbm{1}).$ 
Employing the intermediate group ${\GL}_n(D)^+:={\SL}_n(D) \cdot {\GL}_n(E)$ and its irreducible ${\GL}_n(E)$-distinguished representation $\pi^+$ in ${\Res}^{{\GL}_n(D)}_{{\GL}_n(D)^+} \pi,$ we show that $\pi^+$ is the only such representation with multiplicity one (Lemma \ref{lemma1}). 
We then analyze 
\[
\{\chi \in \widehat{({\GL}_n(D)^+/{\SL}_n(D))} : \pi^+ \otimes \chi = \pi^+ \}
\]
whose cardinality comes out to be 
\[
|Z_\pi|/|Y_\pi|
 = \left(\frac{\dim_{\CC} \rho_{\pi_\flat}}{\mathcal{C}_{\pi}}\right)^2 \cdot \big|\{\tau_\flat \in \Irr({\SL}_n(D)): \tau_\flat \subset \pi^+|_{\SL_n(D)}\}/\sim\big|.
\]
On the other hand, we have the following equality, see Equation \eqref{imp equation},
\begin{align*}
\dim_{\CC}{\Hom}_{\SL_n^*(E)}(\pi,\mathbbm{1}) = &
\dim_{\CC} \rho_{\pi_\flat} ~\cdot~ \big|\{\tau_\flat \in \Irr({\SL}_n(D)): \tau_\flat \subset \pi^+|_{\SL_n(D)}\}/\sim \big| \\
& ~ \cdot~ \dim {\Hom}_{\SL_n^*(E)}(\pi_\flat,\mathbbm{1}).
\end{align*}
Hence, the common factor, $|\{\tau_\flat \in \Irr({\SL}_n(D)): \tau_\flat \subset \pi^+|_{\SL_n(D)}\}/\sim|,$ yields the theorem. \\

In particular, we note that $\mathcal{C}_{\pi} \in \{1, 2\},$ where $\mathcal{C}_{\pi}=1$ unless ${\Res}^{{\GL}_n(D)}_{{\GL}_n(D)^+} \pi \cong \pi^{+} \oplus (\pi^{+} \otimes \eta)$ for some $\eta \in \widehat{\GL_{n}(D)^{+}}$ and $\pi^{+} \in \Irr(\GL_{n}(D)^{+})$ with $\pi^{+} \not\s \pi^{+} \otimes \eta$.
For $n=2$, 
we further provide a necessary condition for the representation $\pi_{\flat} \in \Irr(\SL_{2}(D)$ which is $\SL_2^*(E)$-distinguished. 
This condition is stated in terms of a certain degenerate Whittaker functional with respect to a non-trivial character of $D$, see  Proposition \ref{degenerate}. 
This is similar to a result of  Anadavardhanan-Prasad \cite[Lemma~4.1]{ap18}.
We expect an analogous result to hold for the representations of $\SL_n(D)$ which we are able to achieve only for $\SL_2(D)$.

We apply our main result Theorem \ref{mainthm-intro} to irreducible principle series representations $\pi$ of $\GL_2(D)$. We provide explicit dimensions of ${\Hom}_{\SL_2^*(E)}(\pi_\flat, \mathbbm{1})$ for irreducible representations $\pi_{\flat}$ appearing in ${\Res}_{\SL_2(D)}^{\GL_2(D)} (\pi)$, see Proposition \ref{summarypro}. 
These formulas follows from some computations in Mackey theory and $L$-packets of $\SL_2(D).$ We refer the reader to Section \ref{example} for full details.

In Section \ref{splitcouterpart}, our machinery developed in Section \ref{mainresult} is applied to split cases, where $G={\GL}_{2n}(F), H={\GL}_n(E), G_\flat={\SL}_{2n}(F).$ We then have $H_{\flat} = \SL_{n}^{*}(E)$ which remains the same as in the case of the non-split inner forms. 
However, the fact that the restriction of an irreducible admissible representation of $\GL_{2n}(F)$ to the subgroup $\SL_{2n}(F)$ is multiplicity free (\cite[Theorem 1.2]{tad92}), basically affects the proof so as to have $\mathcal{C}_\pi =1$ and $\dim_{\CC} \rho_{\pi_{\flat}} =1,$ and hence we produce the split counterpart Theorem \ref{mainthem-split}, of our main result in Theorem \ref{mainthm-intro}. 

We finally remark that our study in this present paper will continue towards our forthcoming other non-quasi-split cases and expect such a formula. 
For  $G=\GL_n(D)$ and $H=\GL_n(E)$, the study of distinguished representations of $\SL_{n}(D)$  requires at least some multiplicity one results of $\Hom_{H}(\pi,\mu)$ for an irreducible admissible representation $\pi$ of $G$ and a character $\mu$ of $H$, which is available for a quaternion algebra $D$ over $F$ and expected to have higher multiplicity when $[E:F] > 2$. 
As a result, we restrict ourselves to consider only quaternion algebra $D$ over $F$.
Further, it would be interesting to investigate some relations of this distinction under the Jacquet-Langlands correspondence or under a general endoscopic transfer
and to establish some connections with counting Langlands parameters.

\section*{Acknowledgements}
The authors thank Dipendra Prasad for valuable discussions on this subject.
This work was initiated and finalized during our winter visits in 2023 and 2024 at Institut des Hautes \'{E}tudes Scientifiques and we thank for their inspiring research atmosphere. 
KC is supported by a gift from the Simons Foundation (\#840755).
SPP acknowledges the support of the SERB MATRICS grant (MTR/2022/000782).
\section{Preliminaries} \label{prel}
\subsection{Notation} \label{notations}
Let  $F$ be a finite extension of $\QQ_p.$ Denote by $\bar{F}$ an algebraic closure of $F.$
We denote by $\Gamma$ the absolute Galois group $\Gal(\bar{F} / F)$ and by $W_F$ the Weil group of $F.$ 
Given a connected reductive algebraic group $\bG$ defined over $F,$ the group $\bG(F)$ of $F$-rational points shall be written as $G$ and this is applied to other algebraic groups. 
We denote by $\Irr(G)$ the set of isomorphism classes of irreducible admissible complex representations of $G.$

For $i \in \NN$ and $\Gal(\bar{F} / F)$-module $X,$ we denote by $H^i(F, X) := H^i(\Gal(\bar{F} / F), X)$ the Galois cohomology of $X.$ 
Given a topological group $Y,$ we denote by $Z(Y)$ the center of $Y,$  
by $[Y, Y]$ the derived group of $Y,$ and by $\widehat{Y}$ the group $\Hom_{cont}(Y, \CC^\times)$ of continuous quasi-characters.

\subsection{Restriction} \label{restrictionsection}
This section is devoted to some known results in the restriction of representations. We mainly refer to \cite{gk82, tad92, hs11, choiymulti, choiyzem}.

Recall the notation from Section \ref{notations}. Given $G,$ we let $G_\flat$ be a subgroup of $G$ such that
\begin{equation}  \label{cond on G}
[G_\flat,  G_\flat]=[G,  G] \subseteq G_\flat \subseteq G.
\end{equation}  
For $\sigma_\flat \in \Irr(G_\flat), \sigma \in \Irr(G),$ we recall that
\textit{the multiplicity}, denoted by $ \langle \sigma_\flat, \sigma \rangle_{G_\flat},$ of $\sigma_\flat$ in the restriction ${\Res}_{G_\flat}^{G} (\sigma)$ of $\sigma$ to $G_{\flat}$ is defined by
\[
\langle \sigma_\flat, \sigma \rangle_{G_\flat}:= \dim_{\CC}
\; {\Hom}_{G_\flat}(\sigma_\flat, {\Res}_{G_\flat}^{G} (\sigma)).
\]
Given $\sigma_\flat \in \Irr(G_\flat),$ we have $\sigma \in \Irr(G)$ such that 
\[
\sigma_\flat \subset {\Res}_{G_\flat}^{G}(\sigma).
\]
We shall use both $\Pi_{\sigma_\flat}(G_\flat)$ and $\Pi_{\sigma}(G_\flat)$ for the set of equivalence classes of all irreducible constituents of ${\Res}_{G_\flat}^{G}(\sigma).$ 
Due to \cite[Lemma 2.1]{gk82} and \cite[Proposition 2.4 \& Corollary 2.5]{tad92}, we note that $\Pi_{\sigma_\flat}(G_\flat)$ is finite and independent of the choice  $\sigma \in \Irr(G).$ 
For any irreducible constituents $\si_{\flat 1}$ and  $\si_{\flat 2}$ in ${\Res}_{G_\flat}^{G}(\sigma),$ we have $\Pi_{\si_{\flat 1}}(G_\flat) = \Pi_{\si_{\flat 1}}(G_\flat).$ 
From \cite[Lemma 2.1]{gk82}, \cite[Lemma 2.1]{tad92}, and \cite[Chapter 2]{hs11} we have the following decomposition:
\begin{equation} \label{decomp of Res}
{\Res}_{G_\flat}^{G}(\sigma) = \bigoplus _{\tau_\flat \in \Pi_{\sigma}(G_\flat)} \langle \tau_\flat, \sigma \rangle_{G_\flat} ~ \tau_\flat.
\end{equation}
Here 
$\langle \tau_{\flat}, \sigma \rangle_{G_\flat}$
independent of the choice of $\tau_{\flat}$, let us denote this by $\langle \sigma \rangle_{G_\flat}$.
We may rewrite \eqref{decomp of Res} 
\begin{equation} \label{decomp of Res 1}
{\Res}_{G_\flat}^{G}(\sigma) = \langle \sigma \rangle_{G_\flat} \left( \bigoplus _{\tau_\flat \in \Pi_{\sigma}(G_\flat)}  \tau_\flat \right),
\end{equation}

Let $X(\sigma)= \{ \chi \in \widehat{(G/G_\flat)} : \sigma \s \sigma  \otimes \chi \}$. By \cite[Proposition 3.1]{choiymulti}), the cardinality of $X(\sigma)$ has the following property 
\begin{equation} \label{cardinality}
|X(\sigma)| 
= 
{\dim}_{\CC} \; {\Hom}_{G_\flat}({\Res}_{G_\flat}^{G} (\sigma), {\Res}_{G_\flat}^{G} (\sigma))
= 
|\Pi_{\sigma_\flat}(G_\flat)| \cdot \langle \sigma_\flat, \sigma\rangle_{G_\flat}^2.
\end{equation}

Let us review the work of Hiraga-Saito \cite{hs11} in order to describe the restriction from the inner form of ${\GL}_n$ to the inner form of ${\SL}_n$. 
Any $F$-inner forms of ${\GL}_n$ and ${\SL}_n$  are of the form ${\GL}_m(D)$ and ${\SL}_m(D)$ respectively, where $n=md$ and $D$ is a central division algebra of dimension $d^2$ over $F$ (see \cite[Sections 2.2 \& 2.3]{pr94}).
Note that $d=1$ is allowed, in which case $D=F$.
By abuse of notation, we write ${\GL}_m(D)$ and ${\SL}_m(D)$ for the algebraic groups over $F$ and their $F$-points.
Let ${^L}{\GL}_m(D)^\circ$ and ${^L}{\SL}_m(D)^\circ$ denote the connected component of the $L$-group of $\GL_n(D)$ and the $L$-group of $\SL_n(D)$, respectively. 
Since $\Gal(\bar F/F)$ acts trivially on these connected components, we have
\[ 
{^L}{\GL}_m(D)^\circ = {\GL}_n(\CC) ~~~~\text{and}~~~~ {^L}{\SL}_m(D)^\circ  = {\PGL}_n(\CC). 
\]
Due to \cite{abps13, gk82, ht01, he00, hs11, scholze13}, the local Langlands conjecture and the conjectural internal structure of $L$-packets in \cite[Section 4.1]{choiymulti} for all inner forms of ${\SL}_n$ and ${\GL}_n$ are known.

Following \cite[Chapter 12]{hs11}, \cite[Section 3]{abps13} and  \cite[Chapter 9]{art12},
given an $L$-parameter $\varphi_\flat$ of $\SL_m(D),$ 
we denote the centralizer of the image of $\vp$ in $\PGL_n(\CC)$ by $S_{\vp}.$ 
Write $S_{\vp, \scn}$ for the full pre-image of $S_{\vp}$ in the simply connected cover $\SL_n(\CC)$ of 
$\PGL_n(\CC).$
We then have an exact sequence
\begin{equation*} \label{exact isogeny}
1 \longrightarrow \mathbf{\mu_n}(\CC) \longrightarrow S_{\vp, \scn} \longrightarrow S_{\vp} \longrightarrow 1.
\end{equation*}
We set 
\begin{align*}
\cS_{\vp} &:= \pi_0(S_{\vp}) \\
\cS_{\vp, \scn} & := \pi_0(S_{\vp, \scn}) \\
\mathcal{Z}_{\vp, \scn} & := \mathbf{\mu_n}(\CC) / (\mathbf{\mu_n}(\CC)  \cap S_{\vp, \scn}^{\circ}),
\end{align*}
where $\mathbf{\mu_n}(\CC):=\{z \in \CC^\times: z^n=1\},$
and have the exact sequence of connected compoments
\begin{equation*} \label{exact seq for S-gps}
1 \longrightarrow \mathcal{Z}_{\vp, \scn}  \longrightarrow \cS_{\vp, \scn} \longrightarrow \cS_{\vp} \longrightarrow 1.
\end{equation*}
Fix a character $\zeta_{{^L}{\SL}_m(D)^\circ}$ of $\mathbf{\mu_n}(\CC),$ 
which corresponds to the inner form $\SL_m(D)$ of $\SL_n$ via the Kottwitz isomorphism \cite[Theorem 1.2]{kot86}. 
We denote by 
\[
\Irr(\cS_{\vp, \scn}, \zeta_{{^L}{\SL}_m(D)^\circ})
\]
the set of irreducible representations of $\cS_{\vp, \scn}$ 
with central character $\zeta_{{^L}{\SL}_m(D)^\circ}$ on $\mathbf{\mu_n}(\CC).$

Let  $\vp$ be an $L$-parameter of ${\GL}_m(D)$ such that ${pr} \circ \vp=\vp_\flat,$ where $pr: {\GL}_n(\CC) \rightarrow {\PGL}_n(\CC)$ is the the natural projection (also see \cite{weil74, he80}). We have an $L$-packet of ${\GL}_m(D)$ attached to $\vp$ is $\{ \sigma  \}.$  
We then have a bijection
\begin{equation} \label{bij sl}
\Pi_{\vp_{\flat}}\big({\SL}_m(D)\big) \overset{1-1}{\longleftrightarrow} \Irr\big(\cS_{\vp_{\flat}, \scn}({^L}{\SL}_m(D)^\circ, ~ \zeta_{{\SL}_m(D)} \big),
\end{equation}
such that the isomorphism
\begin{equation*} 
\sigma ~ ~ \s 
\bigoplus_{\rho_\flat \in \Irr\big(\cS_{\vp_{\flat}, \scn}({^L}{\SL}_m(D)^\circ), ~\zeta_{{\SL}_m(D)} \big)} \rho_\flat \boxtimes  \si_{\rho_\flat}
\end{equation*}
as representations of $\cS_{\vp_{\flat}, \scn}({^L}{\SL}_m(D)^\circ) \times {\SL}_m(D)$ holds,
where $\si_{\rho_\flat}$ denotes the image of $\rho_\flat$ via the bijection \eqref{bij sl} (see \cite[Lemma 12.6]{hs11}).
It then follows from \cite[p.5]{hs11} that 
\begin{equation} \label{dim=dim=multi}
\langle \sigma_\flat, \sigma \rangle_{{\SL}_m(D)} = \dim \xi_{\si_\flat} = \dim \rho_{\si_\flat}.
\end{equation}
Further, for any $\si_{\flat 1}, \si_{\flat 2} \in \Pi_{\vp_{\flat}}({\SL}_m(D)),$ we have
$
{\dim}\rho_{\si_{\flat 1}} = {\dim}\rho_{\si_{\flat 2}}
$ 
(see also \cite[Remark 3.3]{choiymulti}).
 
We note that \cite[Theorem 4.22 and Remark 4.23]{choiymulti} generalizes the Hiraga and Saito's work in \eqref{dim=dim=multi} to the case of arbitrary connected reductive algebraic groups $\bG$ and $\bG_\flat$ with the condition \eqref{cond on G}.

\section{General set-up and some known cases} \label{section set-up}
\subsection{General setting} 
Recall $G_{\flat} \subset G$ from Section \ref{restrictionsection}. Let $H$ be a subgroup of $G$ and set 
\[
H_{\flat} = H \cap G_{\flat}.
\]
We now consider the following diagram of inclusions
\begin{align} \label{general diagram}
\xymatrix{
G_\flat \ar@{-}[d]_{\cup} \ar@{-}[r]^{\subset} &G\ar@{-}[d]^{\cup}\\
H_\flat \ar@{-}[r]^{\subset}          &H
}
\end{align} 
We will consider the pairs $(G,H)$ for which the dimension of the $\CC$-vector space ${\Hom}_{H}( \pi, \chi)$ is less than or equal to one for all $\pi \in \Irr(G)$ and all characters $\chi$ of $H$.
Then for the pair $(G_{\flat}, H_{\flat})$, we wish to understand  $H_{\flat}$-distinguished representations of $G_{\flat}$.
In our setting, to obtain a multiplicity formula an intermediate subgroup $G^{+} := H \cdot G_\flat \cdot Z(G)$ plays a crucial role. 
Then, we also have the following diagram of inclusions:
\begin{center}
\begin{tikzcd}[column sep={3em}] 
G_\flat
  \arrow[dash]{rr}{\subset} 
  \arrow[dash]{dd}[swap]{\cup} 
  \arrow[dash]{rd}[inner sep=1pt]{\subset} 
& & 
G 
 \arrow[shorten >= 10pt, dash]{dd}[inner sep=1pt]{\cup}
\\
& 
 G^+ 
        \arrow[dash]{ur}[swap,inner sep=1pt]{\subset} 
& &

\\
H_\flat
  \arrow[swap,shorten >= 20pt, dash]{rr}{\subset} 
          \arrow[dash]{ur}[swap,inner sep=1pt]{\subset}
& & 
H
  \arrow[dash]{lu}[swap,inner sep=1pt]{\supset} 
\\
\end{tikzcd}
\end{center} 

The present paper works over the set-up in the diagram \eqref{general diagram}. 
It is also of interest to provide a general set-up in the context of algebraic groups as follows.
Given connected reductive algebraic groups $\bG, \bG_\flat, \bH$ defined over $F$ (possibly, any field $F$) such that $\bG_\flat, \bH < \bG,$ and $[\bG_\flat, \bG_\flat]=[\bG,  \bG],$
one might set $\bH_\flat=\bH \cap \bG_\flat$ and consider groups of $F$-rational points, $G=\bG(F), G_\flat=\bG_\flat(F), H=\bH(F), H_\flat=\bH(F).$ 
The discrepancy  between $\bH_\flat(F)$ and $H \cap G_{\flat}$ and the subtlety between $[\bG_\flat, \bG_\flat]=[\bG,  \bG]$ and $[G_\flat, G_\flat]=[G,  G]$ are noted. 
As one observation, we have the following exact sequence (c.f., \cite[Section 8.1]{springer98} and \cite[Section 5.1.2]{kalrigd15-2}), 
\begin{equation} \label{exact sequence 1}
1 \longrightarrow Z(\bG) \cap \bG_\flat = Z(\bG_\flat) \longrightarrow Z(\bG) \times \bG_\flat \overset{\lambda}{\longrightarrow} \bG \rightarrow 1,
\end{equation}
where the map $\lambda: Z(\bG) \times \bG_\flat \rightarrow \bG$ is given as $\lambda((a,b)) =ab.$
Applying Galois cohomology, we get the following long exact sequence
\begin{equation} \label{exact sequence 2}
1 \rightarrow Z(G) \cap G_\flat \rightarrow Z(G) \times G_\flat \rightarrow G \rightarrow H^1(F, Z(\bG) \cap \bG_\flat) \rightarrow H^1(F, Z(\bG)) \times H^1(F, \bG_\flat) \rightarrow \cdots 
\end{equation}

\begin{rem} \label{keyrem}
\begin{enumerate}
\item When $Z(\bG)$ is a split torus, $H^1(F, Z(\bG))=1.$
\item If $\bG_\flat$ is simply connected, then $H^1(F, \bG_\flat) = 1.$
\item Exact sequences \eqref{exact sequence 1} and \eqref{exact sequence 2} hold for $\bH$ and $\bH_\flat$ in place of $\bG$ and $\bG_\flat,$ respectively.

\item When $Z(\bG)$ is a split torus and when $\bG_\flat$ is simply connected, then  \eqref{exact sequence 2} yields
\[
1 \rightarrow Z(G) \cap G_\flat \rightarrow Z(G) \times G_\flat \rightarrow G \rightarrow H^1(F, Z(\bG) \cap \bG_\flat) \rightarrow 1
\]
and $H^1(F, Z(\bG) \cap \bG_\flat)$ is a finite abelian group.
\end{enumerate}
\end{rem}
\begin{rem}
It appears that in our general setting, one may not have to restrict oneself for the pairs $(G,H)$ for which one has  $\dim_{\CC} {\Hom}_{H}(\pi, \chi) \leq 1$ for $\pi \in \Irr(G)$ and characters $\chi$ of $H$.
If we assume that the pair $(G,H)$ is such that   $\dim_{\CC} {\Hom}_{H}(\pi, \chi)$ is finite for all $\pi \in \Irr(G)$ and characters $\chi$ of $H$, then one can describe $\dim_{\CC} {\Hom}_{H_{\flat}} (\pi_{\flat}, \mathbbm{1})$ in terms of $\dim_{\CC} {\Hom}_{H}(\pi, \mathbbm{1})$ for any $\pi_\flat \in \Irr(G_\flat)$ such that $\pi_\flat \in {\Res}^G_{G_\flat} (\pi).$
\end{rem}

\subsection{Some previous works}

\subsubsection{}
Now we recall the work of Anandavardhanan-Prasad \cite{ap03}.
Let $E/F$ be a quadratic field extension.
We consider $F$-algebraic groups $\bG={\Res}_{E/F} \GL_n, \bH=\GL_n, \bG_\flat={\Res}_{E/F} \SL_n.$
We then have $G=\bG(F)={\GL}_n(E), H=\bH(F)={\GL}_n(F), G_\flat=\bG_\flat(F)={\SL}_n(E)$ then $ H_\flat = \SL_{n}(E) \cap \GL_{n}(F) = {\SL}_n(F).$ Then the diagram \eqref{general diagram} for $n=2$ is the following
\begin{equation} \label{anabnd-prasad}
\xymatrix{
    {\SL}_{2}(E) \ar@{-}[d]_{\cup} \ar@{-}[r]^{\subset} & {\GL}_{2}(E) \ar@{-}[d]^{\cup} \\
    {\SL}_{2}(F) \ar@{-}[r]^{\subset} & {\GL}_{2}(F)
}
\end{equation}
They study $\SL_2(F)$-distinguished representations of $\SL_2(E)$ and give a multiplicity formula \cite[Theorem~1.4]{ap03}. 
Further, they extended their study of $\SL_2(F)$-distinguished representations of ${\SL}_2(E)$ to ${\SL}_n(F)$-distinguished representaations of $\SL_n(E)$ \cite{ap18} with the diagram \eqref{general diagram} as follows
\begin{equation} 
\xymatrix{
    {\SL}_{n}(E) \ar@{-}[d]_{\cup} \ar@{-}[r]^{\subset} & {\GL}_{n}(E) \ar@{-}[d]^{\cup} \\
    {\SL}_{n}(F) \ar@{-}[r]^{\subset} & {\GL}_{n}(F).
}
\end{equation}
Moreover, they also classified the $\SL_{n}(F)$-distinguished representations of $\SL_{n}(E)$.
Let $\pi_\flat$ be an irreducible admissible representation of $\SL_{n}(E).$ 
Then, $\pi_\flat$ is $\SL_{n}(F)$-distinguished if and only if the Langlands parameter of $\pi_\flat$ is in the image of the base change map and $\pi_\flat$ has a Whittaker model with respect to a character $\psi : N(E) \rightarrow \CC^{\times}$ such that $\psi$ is trivial on $N(F)$, where $N$ is the unipotent subgroup consisting of upper triangular matrices with diagonal entries 1, see \cite[Theorem~5.6]{ap18}.
Moreover, they also prove that the multiplicity is also equal to the cardinality of the inverse image of $L$-parameter of $\pi_\flat$ 
under a certain base change map, see \cite[Theorem~5.6]{ap18}.
On the other hand, for $n=2$, Hengfei Lu gave a different proof of these results of Anandavardhanan-Prasad using theta correspondence \cite[Theorem~1.1, Theorem~1.2]{lu18}.

\begin{rem}
The following observations are the main ingredients to obtain a multiplicity formula for distinguished representations for $\SL_2(E)$  \cite[Theorem~1.4]{ap03} and for $\SL_n(E)$ \cite[Propoosition~3.6]{ap18}. 
\begin{enumerate}
\item  For every $\pi_{\flat} \in \Irr({\SL}_n(E))$ such that $\Hom_{{\SL}_n(F)} (\pi_{\flat}, \mathbbm{1}) \neq 0$ there exists $\pi \in \Irr({\GL}_n(E))$ such that $\Hom_{{\GL}_n(F)}(\pi, \mathbbm{1}) \neq 0$ and $\pi_{\flat} \subset \Res^{{\GL}_n(E)}_{{\SL}_n(E)}\pi$.  
\item The $L$-packet of $\pi_{\flat}$ is the set of equivalence classes of irreducible representations appearing in $\Res^{{\GL}_n(E)}_{{\SL}_n(E)}\pi$. The group ${\GL}_n(E)$ acts transitively on the $L$-packet of $\pi_{\flat}$.
\item For any character $\mu : F^{\times} \rightarrow \CC^{\times}$, ${\dim}_{\CC} {\Hom}_{\GL_{n}(F)} (\pi, \mu \circ \det) \leq 1$.
\end{enumerate}
 
\end{rem}

\subsubsection{} \label{conjugation}
Another variation of Anandvardhanan-Prasad case considered by Hengfei Lu in \cite{lu19} which we describe below.
Let $E/F$ be a quadratic field extension. 
Let $D$ denote the quaternion central division algebra over $F$. 
Note that $D_{E} := D \otimes_F E \cong M_{2}(E)$. 
We have a natural embedding $D^{\times} \hookrightarrow D_{E}^{\times} \cong {\GL}_2(E)$. 

We write $\GL_2(D)$ for the non-split inner form of $\GL_2$ as well as its $F$-points by abuse of notation.
We consider $F$-algebraic groups $\bG=\Res_{E/F}\GL_2, \bH=\GL_2(D), \bG_\flat=\Res_{E/F}\SL_2.$
We then have $G=\bG(F)={\GL}_2(E), H=\bH(F)={\GL}_1(D), G_\flat=\bG_\flat(F)={\SL}_2(E)$ then $ H_\flat = \SL_{2}(E) \cap \GL_{1}(D) = {\SL}_1(D).$ Then the diagram \eqref{general diagram} is the following
\begin{equation} \label{anabnd-prasad-variation}
\xymatrix{
    {\SL}_{2}(E) \ar@{-}[d]_{\cup} \ar@{-}[r]^{\subset} & {\GL}_{2}(E) \ar@{-}[d]^{\cup} \\
    {\SL}_{1}(D) \ar@{-}[r]^{\subset} & D^{\times}
}    
\end{equation}
Let $\pi_{\flat}$ be an irreducible representation of $\SL_{2}(E)$ which is $\SL_{1}(D)$-distinguished. 
Then, the dimension ${\dim}_{\CC} {\Hom}_{\SL_{1}(D)} (\pi_{\flat}, \mathbbm{1})$ is explicitly given by \cite[Theorem~1.3]{lu19}. 
It should be noted that an embedding $\SL_1(D) \hookrightarrow \SL_2(E)$ is not unique up to conjugation by elements of $\SL_2(E)$, unlike the embedding of $D^{\times} \hookrightarrow \GL_2(E)$ which is unique up to conjugation by elements of $\GL_2(E)$ by Skolem-Noether theorem. 
Lu fixes an embedding $\SL_1(D) \hookrightarrow \SL_2(E)$ to study the distinguished representations of $\SL_2(E)$ by $\SL_1(D)$. 
His method in obtaining multiplicity formula is different from the work of Anandavardhanan-Prasad and he uses local theta correspondence.

\section{Our case in question} \label{caseunderconsideration}
Let $D$ be a quaternion division algebra over $F$ and $E/F$ is a maximal subfield so that $E/F$ is quadratic. 
We write $\GL_n(D)$ for the non-split inner form of $\GL_n$ as well as its $F$-points by abuse of notation. We apply the same notation to $\SL_n(D).$
We consider $F$-algebraic groups $\bG=\GL_n(D), \bH = \Res_{E/F}\GL_n, \bG_\flat=\SL_n(D).$ 
We consider $G=\bG(F)={\GL}_n(D), H=\bH(F)={\GL}_n(E), G_\flat=\bG_\flat(F)={\SL}_n(D)$ and
\begin{equation} \label{sl*}
H_\flat = G_{\flat} \cap H = {\SL}_n(D) \cap {\GL}_n(E) = \{ g \in {\GL}_n(E) : \mathcal{N}(\det(g))=1 \}.
\end{equation}
Note that $\SL_{n}(E) \subsetneq H_{\flat}$ and we write $H_{\flat} = \SL_{n}^{*}(E)$.
Thus, in this case, the diagram \eqref{general diagram} is the following:
\begin{center} 
\begin{tikzcd}[column sep={3em}] \label{slnD-slnL}
{\SL}_n(D)
  \arrow[dash]{rr}{\subset} 
  \arrow[dash]{dd}[swap]{\cup} 
  \arrow[dash]{rd}[inner sep=1pt]{\subset} 
& & 
{\GL}_n(D) 
 \arrow[shorten >= 10pt, dash]{dd}[inner sep=1pt]{\cup}
\\
& 
{\SL}_n(D) \cdot {\GL}_n(E)  =: {\GL}_n(D)^+ 
        \arrow[dash]{ur}[swap,inner sep=1pt]{\subset} 
& &

\\
{\SL}^*_n(E)
  \arrow[swap,shorten >= 20pt, dash]{rr}{\subset} 
          \arrow[dash]{ur}[swap,inner sep=1pt]{\subset}
& & 
{\GL}_n(E)
  \arrow[dash]{lu}[swap,inner sep=1pt]{\supset} 
\\
\end{tikzcd}
\end{center}
We study $\SL_{n}^{*}(E)$-distinguished representations of $\SL_{n}(D)$ for which we obtain a multiplicity formula.
It should be noted that the embedding $\SL_{n}^{*}(E) \hookrightarrow \SL_{n}(D)$ is not unique up to conjugation by the elements of $\SL_{n}(D)$ (also refer to Subsection \ref{conjugation}). Any two embeddings are conjugate by an element in $\GL_{n}(D)$. The number of different embeddings up to conjugation by elements in $\SL_{n}(D)$ is $| \SL_{n}(D) \backslash \GL_{n}(D)/ \GL_{n}(E)| = |F^{\times} / \mathcal{N}(E^{\times})| =2$.
\begin{rem}
We make the following remarks regarding our case in question:
\begin{enumerate}
\item 
We note that from ${\SL}_n(D) \cdot {\GL}_n(E) Z({\GL}_n(D)) =: {\GL}_n(D)^+$ we can skip writing $Z({\GL}_n(D))$ as it is $F^{\times}$ which is already contained in ${\GL}_n(E)$.

\item ${\GL}_n(D)^+ / \SL_n(D) = {\GL}_n(E)/({\GL}_n(E) \cap \SL_n(D)) ={\GL}_n(E)/{\SL}^*_n(E) \s E^{\times}/E^{1} \s \mathcal{N}(E^\times).$

\end{enumerate}
\end{rem}

\section{Main results} \label{mainresult}

This section is devoted to studying the $\SL_{n}^{*}(E)$-distinguished representations of $\SL_{n}(D)$ as in the case described in Section \ref{caseunderconsideration}. 
We first recall the following recent result due to Hengfei Lu \cite[Theorem~ 1.1, Theorem~5.1]{lu23} which will be used on several occasions.
\begin{thm} \label{mult one of Lu}
Let $\pi$ be an irreducible admissible representation either of $\GL_{n}(D)$ or of $\GL_{2n}(F)$. 
Then, for any character $\mu : E^{\times} \rightarrow \CC^{\times}$, we have 
\[
{\dim}_{\CC} {\Hom}_{{\GL}_n(E)} (\pi, \mu \circ \det) \leq 1. 
\]
\end{thm}

As an extension of \cite[Lemma 3.2]{ap18}, we have
\begin{pro} \label{key pro}
Let $\pi_\flat \in \Irr(\SL_n(D))$ be $\SL_n^*(E)$-distinguished. 
Then there exists a $\pi \in \Irr({\GL}_n(D))$ which is  ${\GL}_n(E)$-distinguished  and
\[
\pi_\flat \subset {\Res}^{{\GL}_n(D)}_{\SL_n(D)} ~ \pi.
\]
\end{pro}
\begin{proof}
Let $\pi \in \Irr({\GL}_n(D))$ be given such that 
\[
\pi_\flat \subset {\Res}^{{\GL}_n(D)}_{\SL_n(D)} (\pi).
\]
For any $\sigma \in \Irr({\GL}_n(E)),$ the vector space of linear functionals
\[
{\Hom}(\pi, \sigma)
\]
is equipped with a natural ${\GL}_n(E)$-action given by
\[
(g\cdot \lambda)(v)= \sigma(g)(\lambda (\pi(g^{-1}) \cdot v)~~~ \text{ for all } v \in \pi, ~~ g\in {\GL}_n(E).
\]
We now consider the $\SL_n^*(E)$-invariant subspace
\[
({\Hom}(\pi, \sigma))^{\SL_n^*(E)} = {\Hom}_{\SL_n^*(E)}(\pi, \sigma).
\]
Recall that $E^1:=\{ a \in E^\times : \mathcal{N}(a)=1 \}$. 
By taking $\sigma= \mathbbm{1}$ the trivial representation of ${\GL}_n(E),$ 
we now have 
\begin{equation} \label{m_alpha}
{\Hom}_{\SL_n^*(E)}(\pi,\mathbbm{1}) = \bigoplus_{\alpha \in \widehat{E^\times/E^1}} m_\alpha \alpha.
\end{equation}
Now using Theorem \ref{mult one of Lu}, ${\dim}_{\CC} {\Hom}_{{\GL}_(E)}(\pi, \alpha) \leq 1$ for any character $\alpha$ of $E^\times$ in (\ref{m_alpha}) we obtain 
\begin{equation} \label{malpha}
m_\alpha \leq 1.
\end{equation}
Since $0 \neq {\Hom}_{\SL_n^*(E)}(\pi_{\flat}, \mathbbm{1}) \subset {\Hom}_{\SL_n^*(E)}(\pi,\mathbbm{1})$, there exists $\alpha \in \widehat{E^{\times}/E^{1}}$ such that $m_{\alpha} =1$.
On the other hand we have $E^{\times}/E^1 \hookrightarrow D^{\times}/SL_1(D) \cong F^{\times}$, therefore there exists a character  $\tilde{\alpha} : D^{\times} \rightarrow \CC^{\times}$ such that $\tilde{\alpha}|_{E^{\times}} = \alpha^{-1}$.
For a character $\tilde{\alpha} \in \widehat{D^{\times}}$ which extends $\alpha^{-1}$, the irreducible representation $\pi \otimes \tilde\alpha$ of $\GL_{n}(D)$ is a ${\GL}_n(E)$-distinguished representation and $\pi_{\flat} \subset {\Res}^{{\GL}_n(D)}_{\SL_n(D)} ~ (\pi \otimes \tilde{\alpha})$. 
\end{proof}

\begin{lm} \label{lemma1}
Let $\pi$ be an irreducible admissible representations of $\GL_{n}(D)$ which is $\GL_{n}(E)$-distinguished. 
Let $m$ be the common multiplicity of an irreducible representation $\pi_{i}^{+}$ of $\GL_{n}(D)^{+}$ such that $\pi_{i}^{+} \subset {\Res}^{{\GL}_{n}(D)}_{{\GL}_{n}(D)^+} \pi$, as in \eqref{decomp of Res 1}.
Write
\begin{equation} \label{eqn in restriction G to G^+}
{\Res}^{{\GL}_{n}(D)}_{{\GL}_{n}(D)^+} \pi \s m \left( \bigoplus_{i} \pi^+_i \right).
\end{equation} 
Then, $m=1$ and there is exactly one $\pi^+_{i}$ which is $\GL_{n}(E)$-distinguished.
\end{lm}
\begin{proof}
Note that $\GL_{n}(E) \subset \GL_{n}(D)^{+}$ therefore
\begin{equation} \label{decom+}
{\Hom}_{\GL_n(E)} (\pi, \mathbbm{1} ) \s m \left( \bigoplus_{i} {\Hom}_{\GL_n(E)} ( \pi^+_i, \mathbbm{1}) \right).
\end{equation}
By Theorem \ref{mult one of Lu}, the ${\Hom}_{\GL_n(E)} (\pi, \mathbbm{1} ) \s \CC$, since $\pi$ is assumed to be $\GL_{n}(E)$-distinguished then Equation \eqref{decom+} gives $m =1$ and it follows that exactly one of the $\pi_{i}^{+}$ is $\GL_{n}(E)$-distinguished.
\end{proof}
\begin{rem} \label{remsize+}
As an alternative way to get $m=1$ in Lemma \ref{lemma1}, we apply \eqref{cardinality}
to get the following
\[
\Big|\{ \chi \in \widehat{{\GL}_n(D)/{\GL}_n(D)^+} : \pi \s \pi \otimes \chi \} \Big| = m^2 \cdot \big|\Pi_{\pi}({\GL}_n(D)^+) \big|.
\]
Since 
\[
{\GL}_n(D)/{\GL}_n(D)^+ \s F^\times/\mathcal{N}(E^\times),
\]
whose order is $2.$ Hence, we have
\[
m=1~\text{ and } \big|\Pi_{\pi}({\GL}_n(D)^+) \big| \leq 2.
\]
\end{rem}

\begin{lm} \label{lemma2}
Let $\pi^{+} \in \Irr({\GL}_n(D)^{+})$ be ${\GL}_n(E)$-distinguished. 
Let $\pi_{\flat}, \pi_{\flat}'$ are contained in $\Res^{{\GL}_n(D)^{+}}_{\SL_n(D)} (\pi^{+})$. 
Then
\[
{\Hom}_{\SL_n^*(E)} (\pi_{\flat},\mathbbm{1}) \cong {\Hom}_{\SL_n^*(E)} (\pi_{\flat}',\mathbbm{1}).
\]
\end{lm}
\begin{proof}
It follows from the fact that there is an element $g \in {\GL}_n(E)$ such that $\pi_{\flat}' \cong \pi_{\flat}^{g}$.   
\end{proof}

\begin{defn} \label{def XYZ}
Let $\pi$ be an irreducible admissible representation of $\GL_{n}(D)$. We define
\[
\begin{array}{cl}
X_\pi :=& \{\alpha \in \widehat{E^\times/E^1}: \pi~\text{ is }~ \alpha-\text{distinguished}\}, \\
 Z_\pi := & \{\chi \in \widehat{F^\times} : \pi \s \pi \otimes \chi \},   \\
 Y_\pi:= &\{\chi \in Z_\pi : \chi|_{\mathcal{N}(E^\times)} = 1\}.
\end{array}
\]
\end{defn}
The following is the main theorem which is an extension of \cite[Proposition 3.6]{ap18}.

\begin{thm} \label{mainthem}
Let $\pi_\flat \in \Irr(\SL_n(D))$ be $\SL_n^*(E)$-distinguished. Choose an ${\GL}_n(E)$-distinguished representation $\pi$ of ${\GL}_n(D)$ such that  
\[
\pi_\flat \subset {\Res}^{{\GL}_n(D)}_{\SL_n(D)} (\pi).
\]
Denote by $\rho_{\pi_\flat} \in \Irr\big(\cS_{\vp, \scn}({^L}{\SL}_n(D)^\circ), ~ \zeta_{{\SL}_n(D)} \big)$ via \eqref{bij sl} the internal structure of $L$-packet of $\SL_n(D),$ (refer to Section \ref{restrictionsection}).
Let  $\mathcal{C}_{\pi}$ be the number of irreducible representations of ${\GL}_n(D)^+$ containing $\pi_\flat$ in ${\Res}^{{\GL}_n(D)}_{{\GL}_n(D)^+} \pi.$
Then, we have
\begin{equation} \label{mainequation}
{\dim}_{\CC} {\Hom}_{\SL_n^*(E)} (\pi_\flat,\mathbbm{1}) = \frac{|X_\pi|}{|Z_\pi/Y_\pi|} \cdot \frac{{\dim} \rho_{\pi_\flat}}{\mathcal{C}_{\pi}^2}.
\end{equation} 
In particular, $\mathcal{C}_{\pi} \in \{1, 2\}$. We have $\mathcal{C}_{\pi}=1$ unless ${\Res}^{{\GL}_n(D)}_{{\GL}_n(D)^+} \pi \cong \pi^{+} \oplus \pi^{+} \otimes \eta$ for some $\eta \in \widehat{\GL_{n}(D)^{+}}$ and $\pi^{+} \in \Irr(\GL_{n}(D)^{+})$ with $\pi^{+} \not\s \pi^{+} \otimes \eta$. 
\end{thm}
\begin{proof}
Using multiplicity one from Theorem \ref{mult one of Lu}, we get
\[
\dim_{\CC}{\Hom}_{\SL_n^*(E)}(\pi,\mathbbm{1}) = |X_\pi|.
\]
We note from Lemma \ref{lemma1} that $\pi^+$ is the only  $\GL_n(E)$-distinguished representation of ${\GL}_n(D)^+$ with multiplicity one in ${\Res}^{{\GL}_n(D)}_{\SL_n(D)} (\pi)$. 
Then, we have the decomposition
\begin{equation} \label{decomposition of pi+}
\pi^+|_{\SL_n(D)} \s \langle \pi^+  \rangle_{\SL_n(D)} 
\left( \bigoplus_{\{\tau_\flat \in \Irr({\SL}_n(D)): \tau_\flat \subset \pi^+|_{\SL_n(D)}\}/\sim}~\tau_\flat \right).
\end{equation}
For the representations $\tau_{\flat} \in \Irr(\SL_{n}(D))$ such that $\tau_{\flat} \subset {\Res}_{\SL_{n}(D)}^{\GL_{n}(D)^{+}} (\pi^{+})$, we have
\[
\big|\{\tau_\flat \in \Irr({\SL}_n(D)): \tau_\flat \subset \pi^+|_{\SL_n(D)}\}/\sim \big| = \big|\{\tau_\flat \in \Irr({\SL}_n(D)): \tau_\flat \subset \pi|_{\SL_n(D)}\}/\sim \big|.
\]
Due to Lemma \ref{lemma2}, we have
\begin{align}
\dim_{\CC}{\Hom}_{\SL_n^*(E)}(\pi,\mathbbm{1}) =&  \dim_{\CC} {\Hom}_{\SL_{n}(D)} (\pi, \pi_{\flat}) ~\cdot~ 
\dim_{\CC} {\Hom}_{\SL_n^*(E)}(\pi_\flat,\mathbbm{1})  \nonumber \\
 & ~ \cdot~ \big|\{\tau_\flat \in \Irr({\SL}_n(D)): \tau_\flat \subset \pi^+|_{\SL_n(D)}\}/\sim \big|.  \label{imp equation}
\end{align}
Since $\dim_{\CC} {\Hom}_{\SL_{n}(D)} (\pi, \pi_{\flat}) = \dim_{\CC} \rho_{\pi_\flat}$ by \eqref{dim=dim=multi}, we get
\begin{align}
\dim_{\CC}{\Hom}_{\SL_n^*(E)}(\pi,\mathbbm{1}) =&  \dim_{\CC} \rho_{\pi_\flat} ~\cdot~ 
\dim_{\CC} {\Hom}_{\SL_n^*(E)}(\pi_\flat,\mathbbm{1})  \nonumber \\
 & ~ \cdot~ \big|\{\tau_\flat \in \Irr({\SL}_n(D)): \tau_\flat \subset \pi^+|_{\SL_n(D)}\}/\sim \big|.  \label{imp equation}
\end{align}
On the other hand, due to \eqref{cardinality}, we have 
\begin{equation} \label{equality}
\dim {\Hom}_{\SL_n(D)}(\pi^+, \pi^+) = |\{\chi \in \widehat{({\GL}_n(D)^+/{\SL}_n(D))} : \pi^+ \otimes \chi = \pi^+ \}|.
\end{equation}
It follows that the right hand side of \eqref{equality} becomes
\[
\frac{|\{\chi \in \widehat{F^\times} : \pi \otimes \chi \s \chi\}|}{|\{\chi \in \widehat{F^\times}: \pi \otimes \chi \s \pi \text{ and } \chi|_{\mathcal{N}(E^\times)} =1 \}|},
\]
which equals
\[
|Z_\pi|/|Y_\pi|.
\]
Therefore, 
\begin{equation}
\dim_{\CC} {\Hom}_{\SL_n(D)}(\pi^+, \pi^+) = |Z_\pi|/|Y_\pi|.
\end{equation}
On the other hand, due to \eqref{cardinality}, \eqref{equality}, and \eqref{decomposition of pi+}, 
we have
\begin{equation} \label{imp equation 2}
\dim_{\CC} {\Hom}_{\SL_n(D)}(\pi^+, \pi^+) = \langle \pi^+ \rangle_{\SL_n(D)}^2 \cdot |\{\tau_\flat \in \Irr({\SL}_n(D)): \tau_\flat \subset \pi^+|_{\SL_n(D)}\}/\sim|.
\end{equation}
Combining and \eqref{imp equation 2} and \eqref{imp equation}, we have
\begin{equation} \label{lastequality}
\begin{array}{ll}
\dfrac{|X_\pi|}{\dim_{\CC} \rho_{\pi_\flat} \cdot \dim {\Hom}_{\SL_n^*(E)}(\pi_\flat,\mathbbm{1})}
& = \big|\{\tau_\flat \in \Irr({\SL}_n(D)): \tau_\flat \subset \pi^+|_{\SL_n(D)}\}/\sim\big|  \\
& = \dfrac{|Z_\pi|/|Y_\pi|}{\langle \pi^+ \rangle_{\SL_n(D)}^2}.
\end{array} 
\end{equation}
Following  Remark \ref{remsize+}, 
if there is no $\eta \in \widehat{{\GL}_n(D)^+}$ such that $\pi^+\eta \subset {\Res}^{{\GL}_n(D)}_{{\GL}_n(D)^+} \pi$ and $\pi^+ \not\s \pi^+ \otimes \eta,$ then  
\[ 
\dim_{\CC} {\Hom}_{\SL_{n}(D)} (\pi, \pi_{\flat})= \dim_{\CC} {\Hom}_{\SL_{n}(D)} (\pi^{+}, \pi_{\flat}). 
\]
Otherwise, ${\Res}^{{\GL}_n(D)}_{{\GL}_n(D)^+} \pi \cong \pi^{+} \oplus \pi^{+} \otimes \eta$ and then 
\[ 
\dim_{\CC} {\Hom}_{\SL_{n}(D)} (\pi, \pi_{\flat})= 2 \cdot \dim_{\CC} {\Hom}_{\SL_{n}(D)} (\pi^{+}, \pi_{\flat}). 
\]
So, we have
\[
\mathcal{C}_{\pi} \cdot \langle \pi^+ \rangle_{\SL_n(D)} =\dim_{\CC} \rho_{\pi_\flat},
\]
where $\mathcal{C}_{\pi} \in \{1, 2\}.$ 
Thus, the equality \eqref{lastequality} yields
\[
{\dim}_{\CC} {\Hom}_{\SL_n^*(E)} (\pi_\flat,\mathbbm{1}) = \frac{|X_\pi|}{|Z_\pi/Y_\pi|} \cdot \frac{{\dim} \rho_{\pi_\flat}}{\mathcal{C}_{\pi}^2}.
\]
Therefore, we have the theorem.
\end{proof}

Now we prove a necessary condition for the representation $\pi_{\flat} \in \Irr(\SL_{2}(D)$ which is $\SL_2^*(E)$-distinguished. 
This condition is stated in terms of a certain degenerate Whittaker functional with respect to a non-trivial character of $D$.
This condition is to a result of  Anadavardhanan-Prasad \cite[Lemma~4.1]{ap18}.

Let $N(D) = \left\{ \left( \begin{matrix} 1 & x \\ 0 & 1 \end{matrix} \right) : x \in D \right\} \subset \GL_2(D)$.  
Recall that the map $D \times D \rightarrow F$ given by $(x,y) \mapsto \trace(xy)$ is a non-degenerate $F$-bilinear map, where $\trace$ is the reduced trace map. 
Therefore, we get $D \cong \widehat{D}$ via the map $a \mapsto \psi_{a}$ where $\psi_{a}(x)= \psi_{0}(\trace(ax))$ for all $x \in D$ and $\psi_{0}$ is a non-trivial character of $F$. 
Note that this identification of $D$ and $\widehat{D}$ depends on the choice of $\psi_{0}$.
\begin{defn}
Let $\psi : N(D) \cong D \rightarrow \CC^{\times}$ be a character. For an irreducible representation $\pi$ of $\GL_2(D)$ or of $\SL_2(D)$, a linear functional $\ell : \pi \rightarrow \CC$ is called an $(N(D), \psi)$-{degenerate} Whittaker functional if 
\[
\ell \left( \pi \left( \begin{matrix} 1 & x \\ 0 & 1 \end{matrix} \right) v \right) = \psi(x) \ell(v)
\]
for all $v \in \pi$ and $x \in D$. 
If $\pi$ admits a non-zero $(N(D), \psi)$-{degenerate} Whittaker functional then we call $\pi$ to be $(N(D), \psi)$-generic.
\end{defn}
\begin{rem}
Let $\pi$ be an irreducible admissible representation of $\GL_2(D)$ and $\psi$ a non-trivial character of $D$.
Since every charcter of $D$ is of the form $\psi_{a}$ for some $a \in D$, it follows that if $\pi$ admits a $(N(D), \psi)$-{degenerate} Whittaker functional for a non-trivial character $\psi$ then it also admits non-zero $(N(D), \psi')$-{degenerate} Whittaker functional for all  non-trivial character $\psi'$.
On the other hand, the same is not true for an irreducible representations of $\SL_2(D)$.
If an irreducible representation $\pi_{\flat}$ of $\SL_2(D)$ admits a non-zero $(N(D), \psi)$-{degenerate} Whittaker functional then it also admits a non-zero $(N(D), \psi_{a})$-{degenerate} Whittaker functional for all $a \in \langle \SL_1(D), D^{\times 2} \rangle$ the subgroup generated by  $\SL_1(D)$ and $D^{\times 2}$, where $D^{\times 2} = \{ x^2 : x \in D^{\times} \}$.
Note that $\langle \SL_{1}(D) \cdot D^{\times 2} \rangle$ is a proper subgroup of $D^{\times}$ and $\langle \SL_{1}(D) \cdot D^{\times 2} \rangle = \mathcal{N}^{-1} (F^{\times 2} )$. 
Moreover, $[D^{\times} : \langle \SL_1(D), D^{\times 2} \rangle] = [F^{\times} : F^{\times 2} ]$.
It follows that an irreducible representation of $\SL_2(D)$ or $\GL_2(D)$, which admits a non-trivial $(N(D), \psi)$-{degenerate} Whittaker functional for some non-trivial character $\psi$, must be infinite dimensional.
\end{rem}
\begin{pro} \label{degenerate}
Suppose $\pi_{\flat} \in \Irr(\SL_{2}(D))$ admits a degenerate Whittaker functional for some non-trivial character of $N(D)$. 
If $\pi_{\flat}$ is distinguished by $\SL_{2}^{*}(E)$, then $\pi_{\flat}$ must be $(N(D), \psi)$-generic for a
character $\psi : D \rightarrow \CC^{\times}$ such that $\psi$ is trivial on $E$.
\end{pro}
\begin{proof}
Let $N(E) = N(D) \cap \SL_{2}^{*}(E)$.
Let $\pi_{\flat, E}$ denote the largest quotient of $\pi_{\flat}$ on which $N(E) \cong E$ operates trivially. 
We have $\pi_{\flat, E} \neq 0$, because $\pi_{\flat}$ is distinguished by $\SL_{2}^{*}(E)$ and $N(E) \subset \SL_{2}^{*}(E)$.
Therefore, $\pi_{\flat, E}$ is a smooth module for $N(D)/N(E) \cong D/E$, which we analyze below.
\begin{enumerate}
\item[Case 1:] Suppose $D/E$ operates non-trivially on $\pi_{\flat, E}$. 
Then there exists a nontrivial character $\psi : D/E \rightarrow \CC^{\times}$ such that $\pi_{\flat}$ is $(N(D), \psi)$-generic.
\item[Case 2:] Suppose $D/E$ operates trivially on $\pi_{\flat, E}$. 
Let $\ell : \pi \rightarrow \mathbbm{1}$ be a non-zero $\SL_{2}^{*}(E)$-invariant form.
Then $N(D)$ operates trivially on $\ell$ and hence $\ell$ is also invariant under $N(D)$.
Therefore, the group generated by $N(D)$ and $\SL_{2}^{*}(E)$ operates trivially on $\ell$.
Let $L$ be the subgroup of $\SL_{2}(D)$ generated by $N(D)$ and $\SL_{2}^{*}(E)$.  
It can be observed that the quotient $\SL_{2}(D)/L$ is a quotient of $E^{1} \backslash \SL_{1}(D)$ which is compact.
Therefore, $\pi$ must be a finite dimensional representation being a smooth representation of some compact group.
A contradiction to $\pi_\flat$ admitting a degenerate Whittaker functional for some non-trivial character of $N(D)$. \qedhere
\end{enumerate}
\end{proof}
\section{An application to  ${\GL}_2(D)$}  \label{example}

Let $D$ be a quaternion division algebra over $F$ and $E/F$ is a maximal subfield so that $E/F$ is quadratic. 
We consider $G={\GL}_2(D), H={\GL}_2(E), G_\flat={\SL}_2(D),$ and
\begin{equation} \label{sl*-2}
H_\flat={\SL}^*_2(E) = {\SL}_2(D) \cap {\GL}_2(E) = \{ g \in {\GL}_2(E) : \mathcal{N}(\det(g))=1 \}.
\end{equation}
Let 
\[
B(D) = \left\{ \left( \begin{matrix} a & b \\ 0 & d \end{matrix} \right) : a,d \in {\GL}_1(D) ~~\&~~ b \in D \right\}.
\]
be the minimal parabolic subgroup of $\GL_{2}(D)$.
Let $\pi \in \Irr({\GL}_2(D))$ be the normalized parabolic induction
\[
{\Ind}_{B(D)}^{{\GL}_2(D)} ~(\chi_1  \otimes \chi_2),
\]
where $\chi_1, \chi_2 : D^{\times} \rightarrow \CC^{\times}$ are characters.
We note that $\chi_1 \chi_2^{-1} \neq |\mathcal{N}|^{\pm 1},$ since $\pi$ is irreducible, see \cite[Lemma 2.5]{ta1990}, and throughout the section, we keep this condition.

Recall that $\widehat{D^{\times}} = \widehat{F^{\times}}$.
Then, the irreducible representations in the restriction of $\pi$ to $\SL_2(D)$ form the $L$-packet containing $\pi_\flat$, and it equals 
\[
\{\pi_{\flat} \} ~\text{ or }
\{\pi_\flat, \pi_\flat'\},
\]
where $\pi_\flat, \pi_\flat'$ are irreducible constituents of 
\begin{equation} \label{restosl(2)}
{\Ind}_{B(D) \cap \SL_2(D)}^{\SL_2(D)} ~ (\chi_1 \otimes \chi_2).
\end{equation}
Here, we note that 
\[
B(D) \cap {\SL}_2(D) =  \left\{ \left( \begin{matrix} a & b \\ 0 & d \end{matrix} \right) : a,d \in {\GL}_1(D) ~~\&~~ \mathcal{N}(ab)=1, b \in D \right\}.
\]
For a character $\chi$ of $F^{\times}$, we use the same letter $\chi$ for the character of $\GL_2(D)$ given by $g \mapsto \chi( \mathcal{N}(g))$. 
Then
\[
\chi \otimes {\Ind}_{B(D)}^{{\GL}_2(D)} ~(\chi_1  \otimes \chi_2) \s {\Ind}_{B(D)}^{{\GL}_2(D)} ~(\chi_1  \otimes \chi_2)
\]
if and only if
\[
\chi\chi_1  \otimes \chi\chi_2 = \chi_1  \otimes \chi_2 ~\text{ or }~ \chi\chi_1  \otimes \chi\chi_2 = \chi_2  \otimes \chi_1
\] 
as $\GL_{1}(D) \times \GL_{1}(D)$-module.
Note that we must have $\chi^2 = \mathbbm{1}$ and $\chi = \chi_1 \chi_2^{-1} = \chi_2 \chi_1^{-1}$.
We get 
\[
Z_\pi 
= \left\{ 
\begin{array}{l l}
    \{ \mathbbm {1}, \chi_1\chi_2^{-1}\}, & \: \text{if} ~ \chi_1 \neq \chi_2, (\chi_1 \chi_2^{-1})^2 = \mathbbm{1}\\

    \{ \mathbbm {1} \}, & \: \text{otherwise.} \\
  \end{array}
  \right.
\]

Recall $Y_\pi=\{\chi \in Z_\pi: \chi|_{\mathcal{N}(E^\times)} =1\}$. 
So, we have
\[
Y_\pi 
= \left\{ 
\begin{array}{l l}
    \{\mathbbm{1}, \omega_{E/F} \}, & \: \text{if} ~ \chi_1\chi_2^{-1}=\omega_{E/F}, \\
    \{ \mathbbm {1} \}, & \: \text{otherwise,} \\
  \end{array}
  \right.
\]
where $\omega_{E/F}$ is the quadratic character of $F^\times$ associated to the quadratic extension $E/F$ via the local class field theory.
\subsection{Mackey theory} \label{mackey}
Let $\pi = {\Ind}_{B(D)}^{{\GL}_2(D)} ~(\chi_1  \otimes \chi_2)$, and $\mu : E^{\times} \rightarrow \CC^{\times}$ with $\mu|_{E^1} = 1$, be such that $\Hom_{{\GL}_2(E)} (\pi, \mu \circ \det) \neq 0$. 
By Mackey theory, for $\pi|_{{\GL}_2(E)}$ we need to understand the double cosets 
\[
B(D) \backslash {\GL}_2(D) / {\GL}_2(E).
\]
Note that $B(D) \backslash {\GL}_2(D)$ is $\mathbb{P}^{1}(D)$.
Then the action of ${\GL}_2(E)$ on $\mathbb{P}^{1}(D)$ has two orbits, one is the closed orbit $\mathbb{P}^{1}(E)= B(E) \backslash {\GL}_2(E)$ and the other is the open orbit $\mathbb{P}^{1}(D) \smallsetminus \mathbb{P}^{1}(E) = D^{\times} \backslash {\GL}_2(E)$. Therefore we get the following short exact sequence
\begin{equation} \label{ses}
0 \rightarrow {\Ind}_{D^{\times}}^{{\GL}_2(E)} (\chi_1  \bar{\chi}_2 )\rightarrow {\Res}_{{\GL}_2(E)}^{{\GL}_2(D)} (\pi) \rightarrow {\Ind}_{B(E)}^{{\GL}_2(E)} (\chi_1 \otimes \chi_2) \rightarrow 0.
\end{equation}
where $\bar{\chi}_2$ denotes the character of $D^{\times}$ which is $\chi_2$ pre-composed with the non-trivial involution on $D$.
Since $\chi$ factors through the reduced norm map $\bar{\chi}_2(x) = \chi_2(x)$ for all $x \in D^{\times}$.
For simplicity, for a character $\mu : E^{\times} \rightarrow \CC^{\times}$ we write the character $\mu \circ \det$ of ${\GL}_2(E)$ by $\mu$. 
By applying the functor ${\Hom}_{{\GL}_2(E)}(- , \mu)$ to the above short exact sequence \eqref{ses}, we get the following long exact sequence
\[
0 \rightarrow {\Hom}_{{\GL}_2(E)} \left( {\Ind}_{B(E)}^{{\GL}_2(E)} (\chi_1 \otimes 
\chi_2), \mu \right) \rightarrow 
{\Hom}_{{\GL}_2(E)} (\pi, \mu) \rightarrow {\Hom}_{{\GL}_2(E)} \left( {\Ind}_{D^{\times}}^{{\GL}_2(E)} (\chi_1 \bar{\chi}_2), \mu \right) 
\]
\begin{equation} \label{les}
\rightarrow {\Ext}^{1}_{{\GL}_2(E)} \left( {\Ind}_{B(E)}^{{\GL}_2(E)} (\chi_1 \otimes \chi_2), \mu \right) \rightarrow \cdots
\end{equation}
Therefore, $\pi$ is $({\GL}_2(E), \mu)$-distinguished if and only if either ${\Ind}_{B(E)}^{{\GL}_2(E)} (\chi_1 \otimes \chi_2)$ is $({\GL}_2(E), \mu)$-distinguished or ${\Ind}_{D^{\times}}^{{\GL}_2(E)} (\chi_1 \bar{\chi}_2)$ is $({\GL}_2(E), \mu)$-distinguished.
The condition $\chi_1 \chi_2^{-1} \neq |\mathcal{N}|^{\pm 1}$ implies that ${\Ind}_{B(E)}^{{\GL}_2(E)} (\chi_1 \otimes \chi_2) $ is irreducible.
Therefore, ${\Hom}_{{\GL}_2(E)} \left( {\Ind}_{B(E)}^{{\GL}_2(E)} (\chi_1 \otimes \chi_2), \mu \right) =0$ for any character $\mu \in \widehat{E^{\times}}$.
By \cite[Corollary~5.9]{p90}, we know that 
\[ 
{\Hom}_{{\GL}_2(E)} \left( {\Ind}_{B(E)}^{{\GL}_2(E)} (\chi_1 \otimes \chi_2), \mu \right) =0 ~~~\text{if and only if}~~~ 
{\Ext}^{1}_{{\GL}_2(E)} \left( {\Ind}_{B(E)}^{{\GL}_2(E)} (\chi_1 \otimes \chi_2), \mu \right) =0.
\]
From the long exact sequence in \eqref{les}, we conclude that 
\[
{\Hom}_{{\GL}_2(E)} (\pi, \mu) \cong {\Hom}_{{\GL}_2(E)} \left( {\Ind}_{D^{\times}}^{{\GL}_2(E)} (\chi_1 \bar{\chi}_2), \mu \right).
\]
Then, by using Frobenius reciprocity, we get
\[
{\Hom}_{{\GL}_2(E)} \left( {\Ind}_{D^{\times}}^{{\GL}_2(E)} (\chi_1 \bar{\chi}_2), \mu \right) = {\Hom}_{D^{\times}} (\chi_1 \bar{\chi}_2, \mu).
\]
Then, we have
\[
X_{\pi} = \left\{ \mu \in \widehat{E^{\times}} : {\Hom}_{{\GL}_2(E)} \left( {\Ind}_{D^{\times}}^{{\GL}_2(E)} (\chi_1 \bar{\chi}_2), \mu \right) \neq 0 \right\} 
= \left\{ \mu \in \widehat{E^{\times}} : {\Hom}_{D^{\times}} (\chi_1 \bar{\chi}_2, \mu) \neq 0 \right\}.
\]
Now, ${\Hom}_{D^{\times}} (\chi_1 \bar{\chi}_2, \mu) \neq 0$ is equivalent to $\mu(\mathcal{N}(d)) = \chi_1(\mathcal{N}(d)) \chi_2(\mathcal{N}(\bar{d})) = (\chi_1 \chi_2) (\mathcal{N}(d))$ for all $d \in D^{\times}$.
Since the reduced norm map $\mathcal{N} : D^{\times} \rightarrow F^{\times}$ is surjective we get that $\mu|_{F^{\times}} = \chi_1 \chi_2$.
Note that $\mu$ is trivial on $E^1$ and $E^{\times}/E^{1} \s \mathcal{N}(E^{\times}) \subset F^{\times}$.
Note that $[E^{\times} : E^{1} F^{\times} ] = [N(E^{\times}) : F^{\times 2}]$.
So the number of $\mu \in \widehat{E^{\times}/E^{1}}$ for which ${\Ind}_{D^{\times}}^{{\GL}_2(E)} (\chi_1 \bar{\chi}_2)$ is $({\GL}_2(E), \mu)$-distinguished is $[E^{\times} : E^{1} F^{\times} ]$.
Moreover, ${\Ind}_{D^{\times}}^{{\GL}_2(E)} (\chi_1 \bar{\chi}_2)$ is $({\GL}_2(E), \mu)$-distinguished if and only if it is $({\GL}_2(E), \mu \mu_{i})$-distinguished, where $\mu_i$ is any character of $E^{\times}/E^{1}F^{\times}$.
Therefore, we conclude that 
\[
|X_{\pi}| = [E^{\times} : E^{1}F^{\times}] = \frac{1}{2}[F^{\times}: F^{\times 2}].
\]

\subsection{Case A: $\chi_1\chi_2^{-1} = \omega_{E/F}$} \label{caseA} The $L$-packet for $\SL_{2}(D)$ determined by the representation $\pi$ of $\GL_{2}(D)$ contains two elements and hence it equals
\[
\{\pi_{\flat}, \pi_{\flat}'\}.
\]
In this case, we have
\[
|Z_\pi| = 2
~~
\text{ and }
~
Y_\pi=Z_\pi.
\]
Due to \eqref{cardinality}, we have
\[
\dim_{\CC} \rho_{\pi_\flat} =\dim_{\CC} \rho_{\pi'_\flat} =1  \text{ and } \mathcal{C}_\pi=1.
\]
As ${\Res}^{\GL_2(D)}_{\SL_2(D)} (\pi) ={\pi}_{\flat} \oplus {\pi}_{\flat}'$, we have 
\[ {\Hom}_{\SL_2^*(E)} (\pi, \mathbbm{1} ) = {\Hom}_{\SL_2^*(E)} (\pi_{\flat}, \mathbbm{1} ) \bigoplus {\Hom}_{\SL_2^*(E)} (\pi_{\flat}', \mathbbm{1} ).
\]
Then, Section \ref{mackey} yields
\[
{\dim}_{\CC} {\Hom}_{\SL_n^*(E)} (\pi_\flat,\mathbbm{1}) + {\dim}_{\CC} {\Hom}_{\SL_n^*(E)} (\pi_\flat',\mathbbm{1}) = |X_{\pi}|.
\] 
It follows from Theorem \ref{mainthem} that if $\pi_{\flat}$ is $\SL_2^*(E)$-distinguished then 
\[
{\dim}_{\CC} {\Hom}_{\SL_n^*(E)} (\pi_\flat,\mathbbm{1}) = |X_{\pi}|=\frac{1}{2}[F^{\times}: F^{\times 2}]
\]
and ${\dim}_{\CC} {\Hom}_{\SL_n^*(E)} (\pi_\flat',\mathbbm{1}) = 0$.
Therefore, we conclude that exactly one of $\pi_\flat$ or $\pi_\flat'$ is $\SL^*_2(E)$-distinguished.
\subsection{Case B: $(\chi_1\chi_2^{-1})^2=\mathbbm{1}$ but $\chi_1\chi_2^{-1} \notin \{ \omega_{E/F}, \mathbbm{1} \}$} 
The $L$-packet for $\SL_{2}(D)$ determined by the representation $\pi$ of $\GL_{2}(D)$ contains two elements and hence it equals
\[
\{\pi_\flat, \pi_\flat'\}.
\]
In this case, we have
\[
|Z_\pi| = 2
~~
\text{ and }
~
Y_\pi=\{ \mathbbm{1} \}.
\]
Likewise, as in Section \ref{caseA},
we have
\[
{\dim}_{\CC} {\Hom}_{\SL_n^*(E)} (\pi_\flat,\mathbbm{1}) + {\dim}_{\CC} {\Hom}_{\SL_n^*(E)} (\pi_\flat',\mathbbm{1}) = |X_{\pi}|.
\] 
It follows from Theorem \ref{mainthem} that if $\pi_{\flat}$ is $\SL_2^*(E)$-distinguished then
\[
{\dim}_{\CC} {\Hom}_{\SL_n^*(E)} (\pi_\flat,\mathbbm{1})= \frac{|X_{\pi}|}{2}.
\] 
This implies that both $\pi_\flat$ and $\pi_\flat'$ are $\SL^*_2(E)$-distinguished and
\[
{\dim}_{\CC} {\Hom}_{\SL_n^*(E)} (\pi_\flat,\mathbbm{1})
= {\dim}_{\CC} {\Hom}_{\SL_n^*(E)} (\pi_\flat,\mathbbm{1})
= \frac{|X_{\pi}|}{2} = \frac{1}{4}[F^{\times}: F^{\times 2}].
\]

\subsection{Case C: Either $\chi_{1} \chi_2^{-1} = \mathbbm{1}$ or $(\chi_1 \chi_2^{-1})^2 \neq \mathbbm{1}$} 
The $L$-packet for $\SL_{2}(D)$ determined by the representation $\pi$ of $\GL_{2}(D)$ contains only one elements and hence it equals
\[
\{\pi_\flat \}.
\]
In this case, we have
\[
Z_\pi = \{ \mathbbm{1} \} = Y_\pi
\]
Likewise, as in Section \ref{caseA},
we have
\[
{\dim}_{\CC} {\Hom}_{\SL_n^*(E)} (\pi_\flat,\mathbbm{1}) = |X_{\pi}|=\frac{1}{2}[F^{\times}: F^{\times 2}].
\] 

\subsection{}
We summarize the observations made in this section in the following proposition.
\begin{pro} \label{summarypro}
Let $\pi_{\flat}$ be an irreducible representation of $\SL_2(D)$ such that there is an irreducible representation $\pi = {\Ind}_{B(D)}^{{\GL}_{2}(D)} (\chi_1 \otimes \chi_2)$ with $\chi_{1}, \chi_{2} \in \widehat{F^\times}$ such that $\pi_{\flat} \subset {\Res}_{\SL_{2}(D)}^{\GL_{2}(D)}(\pi)$.   
Then, we have $\dim \rho_{\pi_{\flat}} =1$ and $\mathcal{C}_{\pi} =1$.
\begin{enumerate}
\item Assume $\chi_1\chi_2^{-1} = \omega_{E/F}$. Then, the $L$-packet of $\SL_{2}(D)$ containing $\pi_{\flat}$ has two elements, say $\{\pi_\flat, \pi_\flat' \}$, and exactly one of them is $\SL_{2}^{*}(E)$-distinguished. 
Moreover, if $\pi_{\flat}$ is $\SL_{2}^{*}(E)$-distinguished then 
\[
{\dim}_{\CC} {\Hom}_{\SL_n^*(E)} (\pi_\flat,\mathbbm{1}) =\frac{1}{2}[F^{\times}: F^{\times 2}].
\]

\item Assume $(\chi_1\chi_2^{-1})^2=\mathbbm{1}$ but $\chi_1\chi_2^{-1} \notin \{ \omega_{E/F}, \mathbbm{1} \}.$ Then, the $L$-packet of $\SL_{2}(D)$ containing $\pi_{\flat}$ has two elements, say $\{\pi_\flat, \pi_\flat' \}$, and both of them are $\SL_{2}^{*}(E)$-distinguished. Moreover, we have
\[
{\dim}_{\CC} {\Hom}_{\SL_n^*(E)} (\pi_\flat,\mathbbm{1})
= {\dim}_{\CC} {\Hom}_{\SL_n^*(E)} (\pi_\flat,\mathbbm{1}) = \frac{1}{4}[F^{\times}: F^{\times 2}].
\]

\item Assume either $\chi_{1} \chi_2^{-1} = \mathbbm{1}$ or $(\chi_1 \chi_2^{-1})^2 \neq \mathbbm{1}.$ Then the $L$-packet of $\SL_{2}(D)$ containing $\pi_{\flat}$ has only one element, namely $\pi_{\flat}$, which is $\SL_2^*(E)$-distinguished. Moreover, we have
\[
{\dim}_{\CC} {\Hom}_{\SL_n^*(E)} (\pi_\flat,\mathbbm{1}) =\frac{1}{2}[F^{\times}: F^{\times 2}].
\] 
\end{enumerate}
\end{pro}

\section{The split counterpart of our main theorem} \label{splitcouterpart}
This section is devoted to applying our machinery in Section \ref{mainresult} to the split forms. 
We consider $F$-algebraic groups $\bG=\GL_{2n}, \bH = \Res_{E/F}\GL_n, \bG_\flat=\SL_{2n}.$ 
Then $G=\bG(F)={\GL}_{2n}(F), H=\bH(F)={\GL}_n(E), G_\flat=\bG_\flat(F)={\SL}_{2n}(F)$ and
\begin{equation} \label{sl*-split}
H_\flat = G_{\flat} \cap H = {\SL}_{2n}(F) \cap {\GL}_n(E) = \{ g \in {\GL}_n(E) : N_{E/F}(\det(g))=1 \},
\end{equation}
where $N_{E/F}$ denote the norm map for the quadratic extension $E/F$.
Note that $H_{\flat} = \SL_{n}^{*}(E)$ is the same as in the case of the non-split inner forms in Section \ref{mainresult}.
So, we have the following diagram:
\begin{center} 
\begin{tikzcd}[column sep={3em}] \label{slnD-slnL-split}
{\SL}_{2n}(F)
  \arrow[dash]{rr}{\subset} 
  \arrow[dash]{dd}[swap]{\cup} 
  \arrow[dash]{rd}[inner sep=1pt]{\subset} 
& & 
{\GL}_{2n}(F) 
 \arrow[shorten >= 10pt, dash]{dd}[inner sep=1pt]{\cup}
\\
& 
{\SL}_{2n}(F) \cdot {\GL}_n(E)  =: {\GL}_{2n}(F)^+ 
        \arrow[dash]{ur}[swap,inner sep=1pt]{\subset} 
& &

\\
{\SL}^*_n(E)
  \arrow[swap,shorten >= 20pt, dash]{rr}{\subset} 
          \arrow[dash]{ur}[swap,inner sep=1pt]{\subset}
& & 
{\GL}_n(E)
  \arrow[dash]{lu}[swap,inner sep=1pt]{\supset} 
\\
\end{tikzcd}
\end{center}
The following lemma can be obtained easily similar to Proposition \ref{key pro} for which we skip the details. 
\begin{lm}
Let $\pi_{\flat} \in \Irr(\SL_{2n}(F))$ be a $\SL_{n}^{*}(E)$-distinguished representation.
Then there exists $\pi \in \Irr(\GL_{2n}(F))$ which is $\GL_{n}(E)$-distinguished and $\pi_{\flat} \subset {\Res}_{\SL_{2n}(F)}^{\GL_{2n}(F)} \pi$.
\end{lm}

Note that the subgroup $\GL_{2n}(F)^{+}$ in the group $\GL_{2n}(F)$ is 2. 
Moreover, the restriction of an irreducible admissible representation of $\GL_{2n}(F)$ to the subgroup $\SL_{2n}(F)$ is multiplicity free (\cite[Theorem 1.2]{tad92}).
Hence we get the following lemma.
\begin{lm}
Let $\pi_{\flat} \in \Irr(\SL_{2n}(F))$ be $\SL_{n}^{*}(E)$-distinguished and $\pi \in \Irr(\GL_{2n}(F)$ such that $\pi_{\flat} \subset {\Res}^{\GL_{2n}(F)}_{\SL_{2n}(F)} \pi$. 
Let $\mathcal{C}_{\pi}$ be the number of irreducible representations of ${\GL}_{2n}(F)^+$ containing $\pi_\flat$ in ${\Res}^{{\GL}_{2n}(F)}_{{\GL}_{2n}(F)^+} \pi.$
Then,    
\[
\mathcal{C}_\pi =1.
\]
\end{lm}
Moreover, we have the following from \cite[Theorem 1.2]{tad92} and \eqref{dim=dim=multi}.
\begin{lm} \label{lmdim1}
Let $\pi_{\flat} \in \Irr(\SL_{2n}(F))$. Then, $\dim_{\CC} \rho_{\pi_{\flat}} =1$.    
\end{lm}

As an analog of Lemma \ref{lemma1} for the split case, given $\pi_{\flat} \in \Irr(\SL_{2n}(F))$ be a $\SL_{n}^{*}(E)$-distinguished representation, exactly one of the $\pi_{i}^{+}$ in the restriction
\[
{\Res}^{{\GL}_{2n}(F)}_{{\GL}_{2n}(F)^+} \pi \s  \bigoplus_{i} \pi^+_i
\]
is $\GL_{n}(E)$-distinguished, since $\pi$ is assumed to be $\GL_{n}(E)$-distinguished and by Theorem \ref{mult one of Lu}. 
Similar to Lemma \ref{lemma2}, for any $\pi_\flat, \pi_\flat' \in \Res^{{\GL}_{2n}(F)^{+}}_{\SL_{2n}(F)} (\pi^{+}),$
as $\pi^{+} \in \Irr({\GL}_n(D)^{+})$ is ${\GL}_n(E)$-distinguished and $\pi_{\flat}'$ and $\pi_{\flat}^{g}$ are conjugate by $g \in {\GL}_n(E),$ we have 
\[
{\Hom}_{\SL_n^*(E)} (\pi_{\flat},\mathbbm{1}) \cong {\Hom}_{\SL_n^*(E)} (\pi_{\flat}',\mathbbm{1}).
\]
Due to Lemma \ref{lmdim1} and $\langle \pi^+ \rangle_{\SL_{2n}(F)}=1,$ similar to \eqref{lastequality}, we thus have the following split counterpart of our main theorem (Theorem \ref{mainthem}) for which we omit the details.
\begin{thm} \label{mainthem-split}
Let $\pi_\flat \in \Irr({\SL}_{2n}(F))$ be $\SL_n^*(E)$-distinguished. Choose an ${\GL}_n(E)$-distinguished representation $\pi$ of ${\GL}_{2n}(F)$ such that  
\[
\pi_\flat \subset {\Res}^{{\GL}_{2n}(F)}_{{\SL}_{2n}(F)} (\pi).
\]
Then, we have
\begin{equation*} \label{mainequation-split}
{\dim}_{\CC} {\Hom}_{\SL_n^*(E)} (\pi_\flat,\mathbbm{1}) = \frac{|X_\pi|}{|Z_\pi/Y_\pi|}.
\end{equation*} 
\end{thm}


\end{document}